%% file: Frobenius-Green-v4.tex
\def\version{11/10/2015 version 4
\hfill \href{http://arxiv.org/abs/1208.1746}{arXiv:1208.1746}
}

\documentclass[a4paper,11pt]{amsart}
\usepackage{amssymb,amscd,amsxtra}
\usepackage[numeric,lite]{amsrefs}
\usepackage{fullpage}
\usepackage[all]{xy}
\usepackage{mathrsfs}
\usepackage{hyperref}
\usepackage{stmaryrd}

\usepackage{braket}

\usepackage{pigpen}

\def\PB{\text{\pigpenfont J}}

\usepackage{texdraw}
\input{txdtools}
\input{bukser}

\newdir{ >}{{}*!/-8pt/@{>}}

\theoremstyle{plain}
\newtheorem{thm}{Theorem}[section]

\newtheorem{lem}[thm]{Lemma}
\newtheorem{prop}[thm]{Proposition}
\newtheorem{cor}[thm]{Corollary}
\theoremstyle{definition}
\newtheorem{rem}[thm]{Remark}

\newtheorem*{Assump}{Assumption}
\numberwithin{equation}{section}

\def\ie{\emph{i.e.}}
\def\ds{\displaystyle}
\def\:{\colon}
\def\.{\cdot}
\def\<{\langle}
\def\>{\rangle}
\def\({\left(}
\def\){\right)}
\def\ph#1{\phantom{#1}}
\def\epsilon{\varepsilon}
\def\phi{\varphi}
\def\leq{\leqslant}
\def\geq{\geqslant}
\def\lla{\longleftarrow}
\def\la{\leftarrow}
\def\lra{\longrightarrow}

\def\ra{\rightarrow}
\def\bar#1{\overline{#1}}
\def\hat#1{\widehat{#1}}
\def\tilde#1{\widetilde{#1}}
\def\iso{\cong}

\DeclareMathOperator{\im}{im}

\def\F{\mathbb{F}}

\def\k{\Bbbk}
\def\Q{\mathbb{Q}}

\def\Z{\mathbb{Z}}

\def\ideal{\triangleleft}

\DeclareMathOperator{\Hom}{Hom}
\DeclareMathOperator{\inc}{inc}

\DeclareMathOperator{\rad}{rad}

\def\id{\mathrm{id}}
\DeclareMathOperator*{\colim}{colim}

\DeclareMathOperator{\aug}{aug}
\DeclareMathOperator{\Char}{char}
\DeclareMathOperator{\ev}{ev}

\DeclareMathOperator{\soc}{soc}

\DeclareMathOperator{\ind}{ind}

\DeclareMathOperator{\res}{res}

\DeclareMathOperator{\Tr}{Tr}
\DeclareMathOperator{\unit}{unit}
\DeclareMathOperator{\quo}{quo}
\DeclareMathOperator{\Spec}{Spec}
\def\CAlg{\mathscr{CA}}
\def\Mod{\mathscr{M}}

\begin{document}
\title[Frobenius Green functors]
{Frobenius Green functors}
\author{Andrew Baker}
\address{School of Mathematics \& Statistics,
University of Glasgow, Glasgow G12 8QW, Scotland.}
\email{a.baker@maths.gla.ac.uk}
\urladdr{\href{http://www.maths.gla.ac.uk/~ajb}
                       {http://www.maths.gla.ac.uk/~ajb}}
\date{\version}

\maketitle

\section*{Introduction}

These notes provide an informal introduction to a type of Mackey
functor that arises naturally in algebraic topology in connection
with Morava $K$-theory of classifying spaces of finite groups.
The main aim is to identify key algebraic aspects of the Green
functor structure obtained by applying a Morava $K$-theory to
such classifying spaces. This grew out of joint work with Birgit
Richter~\cite{AB&BR:GaloisLT(BG)}. Of course the Morava $K$-theory
and Lubin-Tate theory of such classifying spaces were important
amongst subjects of~\cites{HKR:K*BG,HKR:GenChar}, and our work
is very much a footnote to those from a topological viewpoint,
but we hope the algebraic structures encountered here are of
wider interest. Since the potential audience is varied, we have
tried to provide necessary background material

For each prime $p$ and each natural number~$n$ there is
a multiplicative cohomology theory $K(n;p)^*(-)=K(n)^*(-)$
defined on spaces which on a point takes the value
\[
K(n)^*(\mathrm{point}) = K(n)^* = \F_p[v_n,v_n^{-1}],
\]
where $v_n\in K(n)^{2-2p^n}$. If $p$ is odd, the values on
spaces are graded commutative $K(n)^*$-algebras, while for
$p=2$ there is a mild deviation from graded commutativity,
namely for a space $X$, and $u,v\in K(n)^*(X)$ of odd degree,
\[
uv-vu = v_n Q(u)Q(v)
\]
for a certain cohomology operation $Q$ on $K(n)^*(-)$; this
is referred to as \emph{quasi-commutativity} in~\cite{NPS:K(N)duality}.
In these notes we focus on aspects of the structure of
$K(n)^*(BG)$ where this complication does not significantly
affect our results.

In practise, it is useful to modify $K(n)^*(-)$ to obtain
a $2$-periodic theory $K_n^*(-)$ where
\[
K_n^*(\mathrm{point}) = K_n^* = \F[u,u^{-1}],
\]
where $u\in K_n^{-2}$ and $\F$ is either $\F_{p^n}$ or
$\bar{\F}_p$. For this theory, the grading is essentially
over $\Z/2$ (or $\Z/2(p^n-1)$) so the ideas of~\cite{CGO}
may prove useful. There is a Galois-theoretic relationship
between $K_n^*(-)$ and $K(n)^*(-)$, corresponding to extending
from $\F_p$ to $\F$, and adjoining a $(p^n-1)$-st root of
$v_n$; therefore passage in either direction between these
theories is well-understood. Note that $K_n^*$ is a graded
field in the obvious sense, so computations are often
simplified with the aid of a strict K\"unneth formula for
products:
\[
K_n^*(X\times Y) \iso K_n^*(X)\otimes_{K_n^*}K_n^*(Y).
\]

Remarkably, despite $BG$ being a large space (for example
it can usually be modelled by an infinite dimensional CW
complex or manifold) there is a finiteness result due to
Ravenel~\cite{DCR:K^*BG}: for every finite group~$G$,
$K_n^*(BG)$ is finite dimensional over $K_n^*$. This is
one of the key foundational results, the other being an
observation that $K_n^*(BG)$ is self-dual as a
$K_n^*(BG)$-module and therefore it is self-injective.
Since $K_n^*(BG)$ is also a local $K_n^*$-algebra, this
means that $K_n^*(BG)$ is a local Frobenius algebra (and
the choice of Frobenius structure is in some sense
functorial). This circle of ideas was explored by
Strickland~\cite{NPS:K(N)duality}. An important observation
(whose proof I learnt from Nick Kuhn) is that
$\Tr^G_1(1)\neq0$, and it easily follows that $\Tr^G_1(1)$
is a basis element for the socle of $K_n^*(BG)$.

Using properties of transfers for covering spaces, we can
consider $K_n^*(B-)$ as a Green functor on subgroups of a
fixed finite group~$G$. In fact, it extends to a globally
defined Green functor on all finite groups. Here we exploit
the fact that $K_n^*(BG)=K_n^*$ precisely when $p\nmid|G|$
to produce pushforward maps $(B\alpha)_*\:K_n^*(BH)\lra K_n^*(BK)$
whenever $\alpha\:H\lra K$ is a homomorphism for which
$p\nmid|\ker\alpha|$. However, Dwyer~\cite{WGD:TransferMaps} 
shows that for Morava $K$-theory at a prime~$p$, the latter 
condition is unnecessary; we will return to this case in 
a future paper.

In Appendix~\ref{sec:FrobAlg} we provide a brief review of
material on local Frobenius algebras.

\section{Recollections on Mackey and Green functors}
\label{sec:Mackey&GreenFunctors}

We refer to Webb~\cite{PW:MackeyGuide} for a convenient
general overview, other useful references are
Bouc~\cites{Bouc:LNM1671,Bouc:LNM1990}.

Let $R$ be a commutative ring (in practise in our work
it will be a field).

A \emph{Mackey functor} $\boldsymbol{M}$ on the subgroups
of a finite group $G$ and taking values in the category
of left $R$-modules ${}_R\Mod$, is an assignment
\[
(H\leq G) \mapsto \boldsymbol{M}(H)\in{}_R\Mod
\]
together with morphisms
\[
\res^H_K\: \boldsymbol{M}(H)\lra \boldsymbol{M}(K),
\quad
\ind^H_K\:\boldsymbol{M}(K)\lra \boldsymbol{M}(H),
\quad
c_g\:\boldsymbol{M}(H)\lra \boldsymbol{M}(gHg^{-1})
\]
for $K\leq H\leq G$ and $g\in G$ which satisfy the following
axioms.
\begin{enumerate}
\item[(\textbf{MF1})]
For $H\leq G$ and $h\in H$,
\[
\res^H_H=\ind^H_H=c_h=\id\:\boldsymbol{M}(H)\lra\boldsymbol{M}(H).
\]
\item[(\textbf{MF2})]
For $L\leq K\leq H\leq G$,
\[
\res_L^K\res_K^H = \res_L^H,
\quad
\ind^H_K\ind^K_L = \ind^H_L.
\]
\item[(\textbf{MF3})]
For $g_1,g_2\in G$ and $H\leq G$,
\[
c_{g_1}c_{g_2}=c_{g_1g_2}\:\boldsymbol{M}(H)\lra\boldsymbol{M}(g_1g_2Hg_2^{-1}g_1^{-1}).
\]
\item[(\textbf{MF4})]
For $K\leq H\leq G$ and $g\in G$,
\[
\res^{gHg^{-1}}_{gKg^{-1}}c_g = c_g\res^H_K,
\quad
\ind^{gHg^{-1}}_{gKg^{-1}}c_g = c_g\ind^H_K.
\]
\item[(\textbf{MF5})]
(\emph{Mackey double coset/decomposition formula})
For $H\leq G$ and $K\leq H\geq L$,
\[
\res^H_L\ind^H_K =
\sum_{g\:L\backslash G/K}
    \ind^L_{L\cap gKg^{-1}}c_g\res^K_{g^{-1}Lg\cap K},
\]
where the sum is over a complete set of representatives
for the set of double cosets $L\backslash G/K$.
\end{enumerate}

Such a Mackey functor $\boldsymbol{A}$ is a \emph{Green functor}
if furthermore $\boldsymbol{A}(H)$ is an $R$-algebra for $H\leq G$,
and the following are satisfied.
\begin{enumerate}
\item[(\textbf{GF1})]
For $K\leq H\leq G$ and $g\in G$, $\res^H_K$ and $c_g$ are
$R$-algebra homomorphisms.
\item[(\textbf{GF2})]
(\emph{Frobenius axiom})
For $K\leq H\leq G$, $x\in\boldsymbol{A}(K)$ and $y\in\boldsymbol{A}(H)$,
\[
\ind^H_K(x\res^H_K(y)) = \ind^H_K(x)y,
\quad
\ind^H_K(\res^H_K(y)x) = y\ind^H_K(x).
\]
\end{enumerate}
\begin{rem}\label{rem:FrobAx-Module}
Notice that if $K\leq H\leq G$, $\boldsymbol{A}(K)$ becomes
both a left and a right $\boldsymbol{A}(H)$-module via $\res^H_K(y)$,
so that for $x\in\boldsymbol{A}(K)$ and $y\in\boldsymbol{A}(H)$,
\[
y\.x = \res^H_K(y)x, \quad x\.y = x\res^H_K(y).
\]
Then the Frobenius axiom simply asserts that
$\ind^H_K\:\boldsymbol{A}(K)\lra\boldsymbol{A}(H)$ is both
a left and a right $\boldsymbol{A}(H)$-module homomorphism.
\end{rem}

Now suppose that $\mathcal{X},\mathcal{Y}$ are two collections
of finite groups where $\mathcal{X}$ satisfies the following
conditions. We say that $K$ is a \emph{section} of a group~$G$
if there is a subgroup $H\leq G$ and an epimorphism $H\lra K$.
\begin{itemize}
\item
If $G\in\mathcal{X}$ and $K$ is a section of $G$, then
$K\in\mathcal{X}$.
\item
Let $G',G''\in\mathcal{X}$. If
\[
1\ra G'\lra G\lra G''\ra 1,
\]
is a short exact sequence, then $G\in\mathcal{X}$.
\end{itemize}

A \emph{globally defined Mackey functor} with respect
to $\mathcal{X},\mathcal{Y}$ on finite groups and
taking values in $R$-modules, is an assignment of an
$R$-module $\boldsymbol{M}(G)$ to each finite group~$G$,
for each homomorphism $\alpha\:G\lra H$ with
$\ker\alpha\in\mathcal{X}$ a homomorphism
$\alpha^*\:\boldsymbol{M}(H)\lra\boldsymbol{M}(G)$,
and for each homomorphism $\beta\:K\lra L$ with
$\ker\alpha\in\mathcal{Y}$ a homomorphism
$\beta_*\:\boldsymbol{M}(K)\lra\boldsymbol{M}(L)$,
satisfying the following conditions.
\begin{enumerate}
\item[(\textbf{GD1})]
When these are defined, $(\alpha_1\alpha_2)^* = \alpha_2^*\alpha_1^*$
and $(\beta_1\beta_2)_* = (\beta_1)_*(\beta_2)_*$.
\item[(\textbf{GD2})]
If $\gamma\:G\lra G$ is an inner automorphism, then
$\gamma^*=\id=\gamma_*$.
\item[(\textbf{GD3})]
Given a pullback diagram of finite groups
\[
\xymatrix{
J \ar@{->>}[dd]_\beta\ar@{ >->}[rr]^\gamma
                             && G\ar@{->>}[dd]^\alpha  \\
                             && \\
K\ar@{ >->}[rr]^\delta && H\ar@{}[uull]|<<<<<<<<<<<<{\PB}
}
\]
the following hold:
\begin{itemize}
\item
if $\ker\alpha=\ker\beta\in\mathcal{Y}$, then $\delta^*\alpha_* = \beta_*\gamma^*$;
\item
if $\ker\alpha=\ker\beta\in\mathcal{X}$, then $\alpha^*\delta_* = \gamma_*\beta^*$.
\end{itemize}
\item[(\textbf{GD4})]
For a commutative diagram of epimorphisms of finite groups
\[
\xymatrix{
G\ar@{->>}[dd]_\alpha\ar@{->>}[rr]^\beta && K\ar@{->>}[dd]^\gamma \\
   && \\
H\ar@{->>}[rr]^(.4)\delta && G/\ker\alpha\ker\beta
}
\]
with $\ker\alpha\in\mathcal{Y}$ and $\ker\beta\in\mathcal{X}$,
we have $\alpha_*\beta^* = \delta^*\gamma_*$.
\item[(\textbf{GD5})]
(Mackey formula)
For subgroups $H\leq G \geq K$ with inclusion maps $\iota^G_H\:H\lra G$,
etc, and $c_g$ induced by $g^{-1}(-)g\:K\cap gHg^{-1}\lra g^{-1}Kg\cap H$,
\[
(\iota^G_K)^*(\iota^G_H)_* =
\sum_{g\:K\backslash G/H}
     (\iota^K_{K\cap gHg^{-1}})_*c_g(\iota^H_{g^{-1}Kg\cap H})^*.
\]
\end{enumerate}

Such a globally defined Mackey functor $\boldsymbol{A}$ is a
\emph{globally defined Green functor} if
\begin{enumerate}
\item[(\textbf{GD6})]
Whenever $\alpha\:G\lra H$ is a homomorphism for which $\alpha^*$
is defined, then $\alpha^*\:\boldsymbol{A}(H)\lra\boldsymbol{A}(G)$
is an $R$-algebra homomorphism.
\item[(\textbf{GD7})]
(Frobenius axiom)
Suppose that $\beta\:K\lra L$ is a homomorphism for which $\beta^*$
and $\beta_*$ are both defined. Note that $\beta^*$ induces left and
right $\boldsymbol{A}(H)$-module structures on $\boldsymbol{A}(G)$
(these coincide when $\boldsymbol{A}(G)$ is commutative). Then
$\beta_*\:\boldsymbol{A}(G)\lra\boldsymbol{A}(H)$ is a homomorphism
of left and right $\boldsymbol{A}(H)$-modules.
\end{enumerate}

For our purposes we will take $\mathcal{X}$ to consist of 
all finite groups, and $\mathcal{Y}$ to consist of either 
trivial groups or all groups of order not divisible by some 
prime $p>0$.

\section{Local Artinian Green functors and globally defined extensions}
\label{sec:GreenFunctors}

Suppose that $\boldsymbol{A}$ is a Green functor on the subgroups
of $G$ taking values in the category of local Artinian $\k$-algebras,
where~$\k$ is a field with $\Char\k=p\geq0$. For $H\leq G$, we will
write $\mathfrak{m}(H)\ideal\boldsymbol{A}(H)$ for the unique maximal
left ideal of $\boldsymbol{A}(H)$; this agrees with the Jacobson
radical, $\mathfrak{m}(H)=\rad\boldsymbol{A}(H)$, which is a two-sided
ideal. If $K\leq H\leq G$, then $\res^H_K$ is a local algebra
homomorphism, \ie, $\res^H_K\mathfrak{m}(H)\subseteq\mathfrak{m}(K)$.
We will call such a Green functor \emph{local Artinian}.

More generally, we can suppose that $\boldsymbol{A}$ is a globally
defined Green functor taking values in the category of local Artinian
$\k$-algebras. In the notation of~\cite{PW:MackeyGuide}*{section~8},
we will take $\mathcal{X},\mathcal{Y}$ to consist of suitable
collections of finite groups. Then the restriction/inflation
homomorphism $\alpha^*\:\boldsymbol{A}(H)\lra\boldsymbol{A}(G)$
associated to a homomorphism $\alpha\:G\lra H$ with
$\ker\alpha\in\mathcal{X}$ is a local algebra homomorphism, and
the induction homomorphism
$\beta_*\:\boldsymbol{A}(K)\lra\boldsymbol{A}(L)$ associated to
a homomorphism $\beta\:K\lra L$ with $\ker\beta\in\mathcal{Y}$
is a left and right $\boldsymbol{A}(L)$-module homomorphism with
respect to the $\boldsymbol{A}(L)$-module structure induced on
$\boldsymbol{A}(K)$ by $\beta^*$. When $\beta$ is the inclusion
of a subgroup we will also write $\beta^*=\res^L_K$ and
$\beta_*=\ind^L_K$.

From now on we assume $\boldsymbol{A}$ is a globally defined
local Artinian Green functor, where we take $\mathcal{X}$ to
consist of all finite groups and $\mathcal{Y}$ to consist of
the trivial groups. One of our goals is to extend $\boldsymbol{A}$
to a Green functor where $\mathcal{Y}$ is enlarged to include
all groups of order not divisible by~$p$.

We make some further assumptions.
\begin{Assump}[\textbf{A}]
For any trivial group $1$, $\boldsymbol{A}(1) = \k$.
\end{Assump}

Notice that for any finite group $G$ and any trivial group~$1$,
the diagram
\[
\xymatrix{
1\ar[r]\ar@/^14pt/[rr]^{=} & G\ar[r] & 1
}
\]
induces
\[
\xymatrix{
\k=\boldsymbol{A}(1)\ar[r]_(.55){\unit}\ar@/^16pt/[rr]^{=}
                & \boldsymbol{A}(G)\ar[r]_(.45){\aug}
                  & \boldsymbol{A}(1)=\k
}
\]
where $\aug=\res^G_1$. Therefore $\ker\aug=\mathfrak{m}(G)$
and
\begin{equation}\label{eq:residue=k}
\boldsymbol{A}(G)/\mathfrak{m}(G) = \k.
\end{equation}

\begin{Assump}[\textbf{B}]
For any finite group $G$, $0\neq\ind^G_1(1)\in\boldsymbol{A}(G)$.
\end{Assump}

When $H\leq K$, we can use the restriction $\res^K_H$ to induce
left and right $\boldsymbol{A}(K)$-module structures on $\boldsymbol{A}(H)$
and by the Frobenius axiom, $\ind^K_H\:\boldsymbol{A}(H)\lra\boldsymbol{A}(K)$
is an $\boldsymbol{A}(K)$-module homomorphism. If $z\in\ker\res^K_H$,
then for any $x\in\boldsymbol{A}(H)$,
\begin{equation}\label{eq:relativesoclecondition}
z\ind^K_H(x) = 0 = \ind^K_H(x)z.
\end{equation}
More generally these remarks apply to the restriction $\alpha^*$
associated to an arbitrary group homomorphism $\alpha$, and
to the associated $\alpha_*$ if it is defined.

We can now give a basic result. Recall that $p\geq0$ is the
characteristic of~$\k$.
\begin{lem}\label{lem:res-p'}
Suppose that $H\leq K$ and $p\nmid|K:H|$. \\
\emph{(a)} $\ind^K_H(1)\in\boldsymbol{A}(K)^\times$. \\
\emph{(b)} $\res^K_H\:\boldsymbol{A}(K)\lra\boldsymbol{A}(H)$
is a split monomorphism. In fact\/ $\ind^K_H\res^K_H$ is
right or left multiplication by a unit in $\boldsymbol{A}(K)$.
\end{lem}
\begin{proof}
(a) By the Mackey double coset formula for the subgroups
$1\leq K\geq H$,
\begin{align*}
\aug(\ind^K_H(1)) = \res^K_1\ind^K_H(1)
   &= \sum_{k\:1\backslash K/H}\ind^1_1\mathrm{c}_k\res^H_1(1) \\
   &= \sum_{k\:1\backslash K/H}1 = |K:H|.
\end{align*}
Thus
\begin{equation}\label{eq:ind(1)}
\ind^K_H(1)-|K:H|\in\mathfrak{m}(K)=\rad\boldsymbol{A}(K),
\end{equation}
so if $p\nmid|K:H|$, $\ind^K_H(1)\in\boldsymbol{A}(K)^\times$. \\
(b) Consider $\ind^K_H\res^K_H\:\boldsymbol{A}(K)\lra\boldsymbol{A}(K)$.
For $x\in\boldsymbol{A}(K)$, the Frobenius axiom gives
\[
x\ind^K_H(1) = \ind^K_H\res^K_H(x) = \ind^K_H(1)x.
\]
By (a) $\ind^K_H(1)$ is invertible in $\boldsymbol{A}(K)$ so
we can use $\ind^K_H(1)^{-1}$ to show that $\ind^K_H\res^K_H$
splits as a left or right $\boldsymbol{A}(K)$-module homomorphism.
\end{proof}

We also have the following.
\begin{prop}\label{prop:ind^G_1}
Let $G$ be a finite group. \\
\emph{(a)} $0\neq\ind^G_1(1)\in\soc\boldsymbol{A}(G)$. \\
\emph{(b)} If\/ $p\nmid|G|$, then\/ $\ind^G_1(1)\in\boldsymbol{A}(G)^\times$. \\

\end{prop}
\begin{proof}
(a) Suppose that
\[
z\in\rad\boldsymbol{A}(G)=\mathfrak{m}(G)=\ker\res^G_1.
\]
By~\eqref{eq:relativesoclecondition}, $z\ind^G_1(1) = 0 = \ind^G_1(1)z$,
so $\ind^G_1(1)\in\soc\boldsymbol{A}(G)$.   \\
(b) This follows from Lemma~\ref{lem:res-p'}(a).
\end{proof}

We also record
\begin{prop}\label{prop:ind^G_P}
If\/ $P$ is a $p$-Sylow subgroup of $G$, then
\[
\res^G_P(\ind^G_1(1)) = |G:P|\ind^P_1(1).
\]
Hence \emph{\textbf{Assumption (B)}} holds for
all finite groups if and only if it holds for
all finite $p$-groups.
\end{prop}
\begin{proof}
The Mackey formula \textbf{(MF5)} gives
\begin{align*}
\res^G_P\ind^G_1(1) &=
\sum_{g\:P\backslash G/1}\ind^P_1\circ c_g\circ\res^1_1(1) \\
&= \sum_{g\:P\backslash G/1}\ind^P_1(1) \\
&= |G:P|\ind^P_1(1) \neq 0.
\end{align*}
The conclusion about \textbf{Assumption (B)} is
obvious.
\end{proof}

We can also deduce some useful results on induction
when the prime~$p$ divides the index.
\begin{prop}\label{prop:ind-pindex}
Suppose that $H\leq G$ and $p\mid|G:H|$. Then for
$z\in\boldsymbol{A}(H)$,
\[
\ind^G_H(z)\in\mathfrak{m}(G).
\]
\end{prop}
\begin{proof}
By the Mackey formula,
\begin{align*}
\res^G_1\circ\ind^G_H(z)
&= \sum_{g\:1\backslash G/H} \ind^1_1\circ c_g\circ\res^H_1(z) \\
&= \sum_{g\:1\backslash G/H} \res^H_1(z) \\
&= |G:H|\res^H_1(z) = 0.
\end{align*}
So $\ind^G_H(z)\in\mathfrak{m}(G)$.
\end{proof}

Here is another important consequence of Lemma~\ref{lem:res-p'}(b)
and Assumption~(A). For a converse, see Theorem~\ref{thm:non-triviality}
which depends on more assumptions on the Green functor $\boldsymbol{A}$.
\begin{prop}\label{prop:p'-groups}
Suppose that $G$ is a group for which $p\nmid|G|$. Then $\boldsymbol{A}(G)=\k$.
\end{prop}




Our next result is crucial in identifying when a Green functor
can be extended to a globally defined Green functor.
\begin{lem}\label{lem:G->G/N}
Suppose that $\Char\k=p>0$. \\
\emph{(a)} Suppose that $P$ is a $p$-group, and that $N$ is a
group for which $p\nmid|N|$ and $P$ acts on $N$ by automorphisms,
so the semi-direct product $PN = P\ltimes N$ is defined. Then
the inclusion $P\lra PN$ induces an isomorphism
\[
\res^{PN}_P\:\boldsymbol{A}(PN)\xrightarrow{\;\iso\;}\boldsymbol{A}(P).
\]
\emph{(b)} Suppose that $K\ideal G$ where $p\nmid|K|$. Then the
quotient epimorphism $\pi\:G\lra G/K$ induces an isomorphism
\[
\pi^*\:\boldsymbol{A}(G/K)\xrightarrow{\;\iso\;}\boldsymbol{A}(G).
\]
\end{lem}
\begin{proof}
(a) Since
\[
|PN:P| = |N/P\cap N| = |N|
\]
and $PN/N\iso P$, Lemma~\ref{lem:res-p'} gives a commutative
diagram
\[
\xymatrix{
\boldsymbol{A}(PN/N)\ar[d]_{\res^{PN/N}_{PN}}\ar@/^15pt/[drr]^(.6){\iso}
                 &&  \\
\boldsymbol{A}(PN)\ar@{ >->}[rr]^(.57){\res^{PN}_P}
   && \boldsymbol{A}(P)\ar@{->>}@/^17pt/[ll]^(.43){\ind^{PN}_P}
}
\]
where $\ind^{PN}_P\res^{PN}_P$ is the identity. It easily follows
that $\res^{PN/N}_{PN}$ is both monic and epic, hence it is an
isomorphism. \\
(b) Let $Q\leq G/K$ be a $p$-Sylow subgroup. The pullback square
\[
\xymatrix{
  & \tilde{Q}\ar[r]\ar[d]
   & Q\ar[d] \\
& G\ar[r] & G/K\ar@{}[ul]|<<<<<<{\PB}
}
\]
defines a subgroup $\tilde{Q}\leq G$ containing $K$ and of order
$|Q|\,|K|$. Let $P\leq\tilde{Q}\leq G$ be a $p$-Sylow subgroup
of $\tilde{Q}$ (and so of~$G$). Then $\tilde{Q}$ is a $p$-nilpotent
group with
\[
\tilde{Q}=PK=P\ltimes K,
\]
and the composition
\[
P\xrightarrow{\;\inc\;}PK=\tilde{Q}\xrightarrow{\;\quo\;} Q
\]
is an isomorphism. By (a), the inclusion induces an isomorphism
\[
\boldsymbol{A}(\tilde{Q})
\xrightarrow[\iso]{\;\inc^*=\res^{\tilde{Q}}_P\;}\boldsymbol{A}(P),
\]
hence $\quo^*\:\boldsymbol{A}(Q)\lra\boldsymbol{A}(\tilde{Q})=\boldsymbol{A}(PK)$
is also an isomorphism. Now we have a commutative diagram
\[
\xymatrix{
\boldsymbol{A}(G/K)\ar[rr]^{\pi^*}\ar@{ >->}[dd]^{\res^{G/K}_Q}
    && \boldsymbol{A}(G)\ar@{ >->}[dd]_{\res^{G}_{PK}}  \\
&& \\
\boldsymbol{A}(Q)\ar[rr]^{\quo^*}_{\iso}\ar@/^17pt/@{->>}[uu]^{\ind^{G/K}_Q}
   && \boldsymbol{A}(PK)\ar@/_17pt/@{->>}[uu]_{\ind^{G}_{PK}}
}
\]
from which it follows that $\pi^*$ is monic and epic, hence
it is an isomorphism.
\end{proof}

This result allows us to define $\alpha_*=(\alpha^*)^{-1}$
for any such homomorphism $\alpha\:G\lra G/K$, and more
generally any homomorphism $\beta\:G\lra H$ with
$p\nmid|\ker\beta|$. Thus we have the following extension
result.
\begin{thm}\label{thm:localArt-extn}
There is a unique extension of $\boldsymbol{A}$ to a globally
defined local Artinian Green functor for the pair
$\mathcal{X},\mathcal{Y}'$, where $\mathcal{Y}'$ consists
of all finite groups of order not divisible by~$p$.
\end{thm}

We now introduce a third condition which is suggested by this
result. When $\boldsymbol{A}(G)$ is Morava $K$-theory of $BG$
at the prime~$p$, this condition is automatic since the groups
for which $\boldsymbol{A}(G)=\k$ are precisely those of order
not divisible by~$p$. More generally, Theorem~\ref{thm:non-triviality}
provides conditions under which this holds. It is not clear if
this always holds for local Artinian Green functors satisfying
Assumptions~(A) and~(B).
\begin{Assump}[\textbf{C}]
Suppose that $K\ideal G$ and $\boldsymbol{A}(K)=\k$, and $\pi\:G\lra G/K$
is the quotient homomorphism. Then $\pi$ induces an isomorphism
$\pi^*\:\boldsymbol{A}(G/K)\xrightarrow{\;\iso\;}\boldsymbol{A}(G)$.
\end{Assump}

%

Now we give a result on the effect of automorphisms of~$G$
on the socle of $\boldsymbol{A}(G)$; when $\boldsymbol{A}(G)$
is a Frobenius algebra (see Section~\ref{sec:Frobenius} for
details), $\dim_\k\soc\boldsymbol{A}(G)$ is $1$-dimensional,
and the full strength of this applies.
\begin{prop}\label{prop:auto-soc}
Suppose that $\Char p>0$. Let $G$ be a finite group and
let $\alpha\:G\lra H$ be an isomorphism. Then
\[
\alpha^*\ind^H_1(1) = \ind^G_1(1).
\]
In particular, if\/ $H=G$ so $\alpha$ is an automorphism
and\/ $\dim_\k\soc\boldsymbol{A}(G)=1$, then
$\alpha^*\:\boldsymbol{A}(G)\lra\boldsymbol{A}(G)$ restricts
to the identity function on $\soc\boldsymbol{A}(G)$.
\end{prop}
\begin{proof}
Taking $\mathcal{X}$ and $\mathcal{Y}$ to consist of all
finite groups and all trivial groups respectively, we can
apply (GD3) to the diagram
\[
\xymatrix{
1\ar[d]_{=}\ar[r] & G\ar[d]^{\alpha}_{\iso} \\
1\ar[r] & H
}
\]
to obtain
\[
\alpha^*\ind^H_1(1) = \ind^G_1(1).
\]

If $\alpha$ is an automorphism, it induces a $\k$-algebra
automorphism $\alpha^*\:\boldsymbol{A}(G)\lra\boldsymbol{A}(G)$
which restricts to an automorphism
$\soc\boldsymbol{A}(G)\lra\soc\boldsymbol{A}(G)$. If
$\soc\boldsymbol{A}(G)=1$ is $1$-dimensional, then by
Proposition~\ref{prop:ind^G_1}(a) it is spanned by
$\ind^G_1(1)$ which is fixed by~$\alpha^*$.
\end{proof}

An automorphism of order $p^e$ must act as the identity
on the $1$-dimensional vector space $\soc\boldsymbol{A}(G)$,
but an automorphism of order not divisible by~$p$ might
be expected to act non-trivially.

\section{The stable elements formula}\label{sec:StableElts}

It is well-understood that analogues of the classic stable
elements formula of Cartan \& Eilenberg~\cite{C&E} often applies
to compute Mackey functors. The very accessible introduction
of Webb~\cite{PW:MackeyGuide}*{section~3} provides the necessary
background material which we will use. We have all the necessary
ingredients for such a result when $\boldsymbol{A}$ is a local
Artinian Green functor which satisfies both of Assumptions~(A)
and~(B).

We assume that $\Char\k=p>0$ and denote by $\mathcal{P}$ the
collection of all non-trivial finite $p$-groups.

First note that for any finite group and a $p$-Sylow subgroup
$P\leq G$, the induction homomorphism
$\ind^G_P\:\boldsymbol{A}(P)\lra\boldsymbol{A}(G)$ is surjective
by Lemma~\ref{lem:res-p'}(b).
\begin{prop}\label{prop:A-projective}
The Mackey functor $\boldsymbol{A}$ on the subgroups of\/ $G$
is $\mathcal{P}$-projective.
\end{prop}
\begin{proof}
The coproduct of the maps $\ind^G_Q$,
\[
(\ind^G_Q)_Q\:\bigoplus_{\substack{Q\leq G\\Q\in\mathcal{P}}}
                                \boldsymbol{A}(Q)\lra\boldsymbol{A}(G),
\]
is surjective since the factor $\ind^G_P$ corresponding to a
$p$-Sylow subgroup is. Now by Dress' theorem~\cite{Dress:Contrib},
see~\cite{PW:MackeyGuide}*{theorem~3.4}, $\boldsymbol{A}$ is a
$\mathcal{P}$-projective Mackey functor.
\end{proof}

Given this result, the theory of resolutions using Amitsur
complexes described in~\cite{PW:MackeyGuide}*{section~3}
can be applied. In particular, we can take the two diagrams
$\mathcal{D}^*,\mathcal{D}_*$ consisting of all morphisms
of the form

\[
\xymatrix{
\boldsymbol{A}(Q_1)\ar[ddr]_{c_g\circ\res^{Q_1}_{Q_1\cap g^{-1}Q_2g}}
& \mathcal{D}^* &
\boldsymbol{A}(Q_2)\ar[ddl]^{\res^{Q_2}_{gQ_1g^{-1}\cap Q_2}} \\
&  & \\
& \boldsymbol{A}(gQ_1g^{-1}\cap Q_2) &
}
\]

\bigskip
\[
\xymatrix{
& \boldsymbol{A}(gQ_1g^{-1}\cap Q_2)
\ar[ddl]_{\ind^{Q_1}_{Q_1\cap g^{-1}Q_2g}\circ c_{g^{-1}}}
\ar[ddr]^{\ind^{Q_2}_{gQ_1g^{-1}\cap Q_2}\ph{abc}} &  \\
&  & \\
\boldsymbol{A}(Q_1) & \mathcal{D}_* & \boldsymbol{A}(Q_2)
}
\]

\bigskip
\noindent
where $Q_1,Q_2\leq G$ with $Q_1,Q_2\in\mathcal{P}$ and $g\in G$.

Then from proposition~(3.6) and corollary~(3.7) of~\cite{PW:MackeyGuide},
we obtain
\begin{prop}[Stable elements formulae]\label{prop:A-stabelts}
For a finite group $G$, $\boldsymbol{A}(G)$ can be computed
from each of the formulae
\[
\lim_{\mathcal{D}^*}\boldsymbol{A} = \boldsymbol{A}(G)
              = \colim_{\mathcal{D}_*}\boldsymbol{A}.
\]
\end{prop}

Since $p$-Sylow subgroups are cofinal in each case, we can take
$Q_1,Q_2$ to be $p$-Sylow subgroups, then as they are all conjugate,
we can consider the equivalent diagrams of morphisms under or over
one particular $p$-Sylow subgroup, say $P\leq G$. This recovers the
formulae
\begin{align}\label{eq:A-stabelts-lim}
\boldsymbol{A}(G) &=
\bigcap_{g\in G}\ker(\res^P_{gPg^{-1}\cap P}-c_g\circ\res^P_{P\cap g^{-1}Pg})\:
                   \boldsymbol{A}(P)\lra\boldsymbol{A}(gPg^{-1}\cap P),   \\
\label{eq:A-stabelts-colim}
\boldsymbol{A}(G) &= \boldsymbol{A}(P)/
\bigcup_{g\in G}\im(\ind^P_{gPg^{-1}\cap P}-\ind^P_{gPg^{-1}\cap P}\circ c_{g^{-1}})\:
                   \boldsymbol{A}(gPg^{-1}\cap P) \lra\boldsymbol{A}(P).
\end{align}

When $N\ideal G$, for each $g\in G$,$c_g$ restricts to an
automorphism of $\boldsymbol{A}(N)$, and if $g\in N$ this
is the identity. Hence $G/N$ acts on $\boldsymbol{A}(N)$
and we can form invariants $\boldsymbol{A}(N)^{GN}$ and
coinvariants $\boldsymbol{A}(N)_{GN}$.  Now we have
\begin{prop}\label{prop:soc-normal}
Suppose that $G$ has the unique normal $p$-Sylow subgroup
$P\ideal G$. Then
\[
\boldsymbol{A}(G) = \boldsymbol{A}(P)^{G/P}
                  = \boldsymbol{A}(P)_{G/P}.
\]
Furthermore,
\[
\res^G_P(\ind^G_1(1)) = |G:P|\ind^P_1(1) \neq 0.
\]
\end{prop}
Of course this can also be proved more directly
by using Proposition~\ref{prop:A-projective} with
the coproduct over all the $p$-Sylow subgroups
of~$G$, or even with just one of them.

\section{Local Frobenius Green functors}\label{sec:Frobenius}

We now assume that our local Artinian Green functor
$\boldsymbol{A}$ satisfies assumptions (A) and (B),
and that $\Char\k=p>0$. We also require
\begin{Assump}[\textbf{QF}]
For each finite group $G$, $\boldsymbol{A}(G)$ is
Frobenius, \ie, it is a finite dimensional
$\k$-algebra and there is an isomorphism of left
$\boldsymbol{A}(G)$-modules
\[
\boldsymbol{A}(G)\xrightarrow{\;\iso\;}
\boldsymbol{A}(G)^* = \Hom_\k(\boldsymbol{A}(G),\k).
\]
\end{Assump}

A choice of such an isomorphism determines a Frobenius
form $\lambda\in\boldsymbol{A}(G)^*$ which is the
element corresponding to~$1$. We then refer to the
pair $(\boldsymbol{A}(G),\lambda)$ as a Frobenius
algebra (structure) on $\boldsymbol{A}(G)$. Such a
Frobenius form is characterized by the requirement
that $\ker\lambda$ contains no non-trivial left (or
equivalently right) ideals. This Frobenius condition
also implies the Gorenstein condition that
$\dim_\k\soc\boldsymbol{A}(G) = 1$; this follows from
the self-duality of $\boldsymbol{A}(G)$ and the
resulting isomorphisms
\[
\Hom_{\boldsymbol{A}(G)}(\k,\boldsymbol{A}(G))
   \iso\Hom_{\boldsymbol{A}(G)}(\boldsymbol{A}(G),\k)
   \iso\Hom_\k(\k,\k)=\k.
\]

Our assumptions (B) and (C) together imply that a linear
form $\lambda\in\boldsymbol{A}(G)^*$ is a Frobenius form
if and only if $\lambda(\ind^G_1(1)) \neq 0$ (since every
non-zero left ideal intersects the socle $\soc\boldsymbol{A}(G)$
non-trivially). We do not require that the choice of Frobenius
form should be contravariantly functorial with respect to~$G$.
However, the element $\ind^G_1(1)$ is covariantly functorial
since if $G\leq H$ then
\[
\ind^H_1(1) = \ind^H_G\ind^G_1(1).
\]

In general, if $\alpha\:G\lra H$ is a homomorphism, the restriction
$\alpha^*\:\boldsymbol{A}(H)\lra\boldsymbol{A}(G)$ need not send
$\soc\boldsymbol{A}(H)$ into $\soc\boldsymbol{A}(G)$, nor need it
be non-zero on it. However, if
$\alpha_*\:\boldsymbol{A}(G)\lra\boldsymbol{A}(H)$ is defined then
it restricts to an isomorphism
\[
\alpha_*\:\soc\boldsymbol{A}(G)\xrightarrow{\;\iso\;}\soc\boldsymbol{A}(H),
\]
since $\alpha_*(\ind^G_1(1)) = \ind^H_1(1)\neq0$.

\begin{lem}\label{lem:FrobForm-f^*}
Suppose that $\lambda$ is a Frobenius form for\/ $\boldsymbol{A}(H)$.
Then\/ $\alpha^*\lambda=\lambda\circ\alpha_*$ is a Frobenius form
for\/ $\boldsymbol{A}(G)$.
\end{lem}
\begin{proof}
This works very generally. Suppose that $A$ and $B$ are local
Frobenius algebras over a field $\k$, and that $f\:A\lra B$ is
a local algebra homomorphism, making $B$ into a left $A$-module
by $a\.b=f(a)b$. Suppose that $f_*\:B\lra A$ is an $A$-module
homomorphism for which $f_*\soc B=\soc A$. Then for any Frobenius
form $\lambda$ on $A$, $f^*\lambda=\lambda\circ f_*$ is a Frobenius
form on $B$. This is easy to see since
\[
f^*\lambda\soc B = \lambda\circ f_*\soc B = \lambda\soc A \neq 0,
\]
and as every non-trivial ideal intersects $\soc B$ non-trivially,
the Frobenius condition holds for the linear form $f^*\lambda$.
\end{proof}
In the situation of the proof, we may define an inner product
on $A$ by $(x\mid y)_A = \lambda(xy)$; similarly, define an
inner product on $B$ by $(x'\mid y')_B = f^*\lambda(x'y')$.
It follows that for $a\in A$ and $b\in B$,
\[
(f(a)\mid b)_B = \lambda\circ f_*(f(a)b)
               = \lambda(af_*(b)) = (a\mid f_*(b))_A.
\]
This is a version of Frobenius reciprocity.

\section{Green functors on abelian groups and $p$-divisible groups}
\label{sec:p-divgps}

For general notions of groups schemes see \cite{WCW:book}; for
$p$-divisible groups see~\cite{LNM302}. A general reference on
formal schemes and formal groups is Strickland~\cite{NPS:spf&fg}.

Let $\k$ be a field of characteristic $\Char\k=p>0$. Recall that
a finite dimensional commutative $\k$-Hopf algebra $H$ represents
a group scheme $\Spec(H)$ is a group valued functor on the category
of commutative $\k$-algebras $\CAlg_\k$ defined by
\[
\Spec(H)(A) = \CAlg_\k(H,A).
\]
Denoting the coproduct by $\psi\:H\lra H\otimes H$ and the product
on $A$ by $\phi\:A\otimes A\lra A$, the group structure is defined
by
\[
f*g = \phi(f\otimes g)\psi
\]
where $f,g\in\CAlg_\k(H,A)$. The unit is given by the counit
$\epsilon\in\CAlg_\k(H,\k)$ and the inverse is given by an algebra
automorphism $\chi\in\CAlg_\k(H,H)$.

We remark that by the Larson-Sweedler theorem, such a Hopf algebra
is a Frobenius algebra.

Recall that $H$ is \emph{connected} if it has no non-trivial idempotents.
In particular, this is true if $H$ is local. Of course this means that
$\Spec(H)$ is connected.

%
A commutative $\k$-algebra valued Green functor $\boldsymbol{A}$
is a \emph{K\"unneth functor} if it takes products of groups to
pushouts of commutative $\k$-algebras, \ie, it satisfies the
strict K\"unneth formula
\[
\boldsymbol{A}(G\times H) = \boldsymbol{A}(G)\otimes\boldsymbol{A}(H)
\]
for every pair of finite groups $G,H$.

We will impose another condition.
\begin{Assump}[\textbf{KF}]
The Green functor $\boldsymbol{A}$ is a K\"unneth functor.
\end{Assump}

When $H=G$, the diagonal homomorphism $\Delta\:G\lra G\times G$
induces the product on $\boldsymbol{A}(G)$. If $G$ is an abelian
group, the multiplication $G\times G\lra G$ is a group homomorphism
and so it induces an algebra homomorphism
\[
\boldsymbol{A}(G) \lra \boldsymbol{A}(G\times G)
                         = \boldsymbol{A}(G)\otimes\boldsymbol{A}(G)
\]
which is coassociative and counital, with antipode induced
by the inverse map $G\lra G$. Furthermore this coproduct is
cocommutative. This shows that $\boldsymbol{A}(G)$ is naturally
a cocommutative Hopf algebra.
\begin{rem}\label{rem:integrals}
Recall the theory of \emph{integrals} for finite dimensional
Hopf algebras, as described in~\cite{SM:HopfAlgonRngs}*{definition~2.1.1}
for example. For a finite dimensional local Hopf algebra $H$,
taking a generator $z\in\soc H$ we have for any $x\in H$,
\[
xz = \aug(x)z = zx,
\]
hence $z$ is a left and right integral for $H$. Therefore
\[
\soc H = \int^{\mathrm{l}}_H = \int^{\mathrm{r}}_H,
\]
so $H$ is unimodular and $\soc H = \int_H$. Of course, if
$H\neq\k$ then $\aug$ is not a Frobenius form.
\end{rem}

Now recall that a sequence of finite abelian group schemes
$\mathcal{G}_r$ ($r\geq1$) where $\mathcal{G}_r$ has order
$p^{rh}$ for some natural number $h\geq1$ over the field~$\k$
forms a \emph{$p$-divisible group} or \emph{Barsotti-Tate
group} of height~$h$ if there are exact sequences of group
schemes fitting into commutative diagrams
\[
\xymatrix{
 \mathcal{G}_r\ar@{ >->}[rr]^{i_{r,r+s}}
    && \mathcal{G}_{r+s} \ar@{.>}[dr]^{p^r}
      \ar@{->>}[rr]^{q_{r+s,s}} && \mathcal{G}_s\ar@{ >.>}[dl] \\
    && &\mathcal{G}_{r+s}&  \\
}
\]
for all $r,s\geq1$, where
$\mathcal{G}_r\xrightarrow{\;i_{r,r+s}\;}\mathcal{G}_{r+s}$
is a kernel for multiplication by $p^{r}$ on $\mathcal{G}_{r+s}$.
These are required to be compatible in the sense that there are
commutative diagrams of the following forms.
\[
\xymatrix{
 \mathcal{G}_r\ar@{ >->}[rr]^{i_{r,r+s}}\ar[ddd]_{=}
 && \mathcal{G}_{r+s}\ar@{ >->}[ddd]_{i_{r+s,r+s+1}} \ar@{.>}[dr]^{p^r}
    \ar@{->>}[rr]^{q_{r+s,s}} && \mathcal{G}_s\ar@{ >.>}[dl]\ar@{ >->}[ddd]^{i_{s,s+1}} \\
     && &\mathcal{G}_{r+s}\ar[d]^{i_{r+s,r+s+1}}&  \\
 && &\mathcal{G}_{r+s+1}&   \\
 \mathcal{G}_r\ar@{ >->}[rr]^{i_{r,r+s+1}} && \mathcal{G}_{r+s+1} \ar@{.>}[ur]^{p^r}
    \ar@{->>}[rr]^{q_{r+s+1,s+1}} && \mathcal{G}_{s+1}\ar@{ >.>}[ul]
}
\]

\bigskip
\[
\xymatrix{
 \mathcal{G}_r\ar@{ >->}[ddd]_{i_{r,r+1}}\ar@{ >->}[rr]^{i_{r,r+s}}
 && \mathcal{G}_{r+s}\ar@{ >->}[ddd]_{i_{r+s,r+s+1}} \ar@{.>}[dr]^{p^r}
    \ar@{->>}[rr]^{q_{r+s,s}} && \mathcal{G}_s\ar@{ >.>}[dl]\ar[ddd]^{=} \\
     && &\mathcal{G}_{r+s}\ar[d]^{pi_{r+s,r+s+1}}&  \\
 && &\mathcal{G}_{r+s+1}&   \\
 \mathcal{G}_{r+1}\ar@{ >->}[rr]^{i_{r,r+s+1}} && \mathcal{G}_{r+s+1} \ar@{.>}[ur]^{p^{r+1}}
    \ar@{->>}[rr]^{q_{r+s+1,s+1}} && \mathcal{G}_s\ar@{ >.>}[ul]
}
\]
If $\mathcal{G}_r=\Spec(H_r)$ for some cocommutative Hopf algebra
$H_r$, then $\dim_\k H_r=p^{rh}$ and there are morphisms of Hopf
algebras
\begin{equation}\label{eq:pdiv-HA}
\xymatrix{
H_r & \ar@{->>}[l] H_{r+s} & \ar@{ >->}[l] H_s
}
\end{equation}
inducing the diagram of group schemes
\begin{equation}\label{eq:pdiv-GpSch}
\xymatrix{
\mathcal{G}_r\ar@{ >->}[r]^{i_{r,r+s}} & \mathcal{G}_{r+s}
               \ar@{->>}[r]^{q_{r+s,s}} &\mathcal{G}_s.
}
\end{equation}

We are interested in the case where each $\mathcal{G}_r$
is connected, and this follows from the requirement that
every $H_r$ is a local $\k$-algebra. Then we have the
following Borel algebra decomposition of each $H_r$.
\begin{prop}\label{prop:Borel-decomp}
For each $r\geq1$, there is an isomorphism of $\k$-algebras
of the form
\[
H_r \iso
\k[x_1,x_2,\ldots,x_\ell]/(x_1^{q_1},x_2^{q_2},\ldots,x_\ell^{q_\ell}),
\]
where $q_i = p^{d_i}$ for some $d_i\geq1$ and $d_1 + d_2 + \cdots + d_\ell = rh$.
\end{prop}
Notice that
\[
\soc H_r = \k\{x_1^{q_1-1}x_2^{q_2-1}\cdots x_\ell^{q_\ell-1}\}.
\]

The restriction epimorphisms $H_r\lra H_{r-1}$ have a limit
\begin{equation}\label{eq:H-lim}
H = \ds\lim_r H_r = \k[\![x_1,x_2,\ldots,x_\ell]\!]
\end{equation}
which has a topological coproduct $\psi\:H\lra H\hat{\otimes}_\k H$
defining a cocommutative \emph{formal group} over $\k$ of
dimension~$\ell$. Each $H_r$ can be recovered from $H$ by
forming the quotient with respect to the ideal generated by
the image of the multiplication by~$p^r$, expressible as a
composition
\[
\xymatrix{
H \ar[r]^(.31){\psi^{(p)}}
  & H\hat{\otimes}_\k\cdots\hat{\otimes}_\k H\ar[r]^(.7){\phi^{(p)}}
  & H
}
\]

We now introduce another assumption.
\begin{Assump}[\textbf{D}]
The Hopf algebras $\boldsymbol{A}(C_{p^r})$ with restriction
and inflation homomorphisms $\res^{C_{p^{r+s}}}_{C_{p^r}}$
and $\res^{C_{p^s}}_{C_{p^{r+s}}}$ (induced by the canonical
quotient homomorphism $C_{p^{r+s}}\lra C_{p^s}$) give rise
to a $p$-divisible group which satisfies
$\mathcal{G}_r=\Spec\boldsymbol{A}(C_{p^r})$.
\end{Assump}

We can immediately deduce a non-triviality result; a version
of this for Morava $K$-theory appeared in John Hunton's PhD
thesis.
\begin{prop}\label{prop:G->Cp}
Suppose that $\pi\:G\lra C_{p^s}$ be an epimorphism
for $s\geq1$. Then the induced homomorphism
$\pi^*\:\boldsymbol{A}(C_{p^s})\lra\boldsymbol{A}(G)$
is a monomorphism.
\end{prop}
\begin{proof}
Lifting a generator of $C_{p^s}$ to $G$, we obtain
a commutative diagram of the form
\[
\xymatrix{
C_{p^{r+s}}\ar[d]\ar@{->>}[rrd]^{\quo} && \\
G \ar@{->>}[rr]^(.415){\pi} && C_{p^{s}}
}
\]
and on applying $\boldsymbol{A}(-)$ this gives
\[
\xymatrix{
\boldsymbol{A}(C_{p^{r+s}}) && \\
\boldsymbol{A}(G)\ar[u]
&& \ar@{ >->}[ll]_(.585){\pi^*}\boldsymbol{A}(C_{p^{s}})\ar@{ >->}[llu]_{\quo^*}
}
\]
where we know from~\eqref{eq:pdiv-HA} that
$\quo^*=\res^{C_{p^{s}}}_{C_{p^{r+s}}}$ is monic and so $\pi^*$
is also monic.
\end{proof}
Of course, if $G$ is a non-trivial $p$-group, such epimorphisms
always exist.
\begin{cor}\label{cor:G->Cp}
If\/ $G$ is a non-trivial $p$-group, then $\boldsymbol{A}(G)$
is non-trivial, $\boldsymbol{A}(G)\neq\k$. In particular,
$0 \neq \ind^G_1(1)\notin\k$.
\end{cor}
\begin{proof}
For the statement about $\ind^G_1(1)$, recall that
$0\neq\ind^G_1(1)\in\soc\boldsymbol{A}(G)$ and since
$\boldsymbol{A}(G)\neq\k$ we must have $\soc\boldsymbol{A}(G)\neq\k$.
\end{proof}
\begin{thm}\label{thm:non-triviality}
Let $G$ be a finite group. Then $\boldsymbol{A}(G)=\k$ if and
only if $p\nmid|G|$.
\end{thm}
\begin{proof}
We must show that if $G$ is a non-trivial finite group whose
order is divisible by $p$, then $\boldsymbol{A}(G)\neq\k$.
We know this holds for $p$-groups, so suppose that $G$ is
not a $p$-group. Let $P\leq G$ be a $p$-Sylow subgroup. Then
the Mackey formula gives
\begin{align*}
\res^G_P\ind^G_1(1) &= \sum_{g\:P\backslash G/1}\ind^P_1 c_g \res^1_1(1) \\
 &= \sum_{g\:P\backslash G/1}\ind^P_1(1) \\
 &= |G:P|\ind^P_1(1).
\end{align*}
By Corollary~\ref{cor:G->Cp}, $|G:P|\ind^P_1(1)\notin\k$,
hence $\ind^G_1(1)\notin\k$. This shows that $\boldsymbol{A}(G)\neq\k$.
\end{proof}

If we restrict attention to finite abelian $p$-groups, then
the following result holds, see~\cite{HKR:K*BG}*{proposition~2.4}
for the Morava $K$-theory version.
\begin{thm}\label{thm:abelgp-mono-epi}
Suppose that $G,H,K$ are finite abelian $p$-groups. \\
\emph{(a)} If $\phi\:G\lra H$ is an epimorphism, then
$\phi^*\:\boldsymbol{A}(H)\lra\boldsymbol{A}(G)$ is monic. \\
\emph{(b)} If $\theta\:K\lra G$ is a monomorphism, then
$\theta^*\:\boldsymbol{A}(G)\lra\boldsymbol{A}(K)$ is
epic.

In each case, the converse also holds.
\end{thm}
\begin{proof}
As a starting point, we recall that Assumption~(\textbf{D})
implies this for the canonical epimorphisms $C_{p^{r+s}}\lra C_{p^s}$
and monomorphisms $C_{p^r}\lra C_{p^{r+s}}$.

\noindent
(a) Since $H$ is isomorphic to a product of cyclic groups,
\[
H \iso C_{p^{s_1}}\times\cdots\times C_{p^{s_k}}
\]
we can choose lifts of the generators to elements of $G$
and then define a homomorphism
\[
C_{p^{r_1+s_1}}\times\cdots\times C_{p^{r_k+s_k}} \lra G
\]
and a factorisation of the canonical quotient
\[
\xymatrix{
C_{p^{r_1+s_1}}\times\cdots\times C_{p^{r_k+s_k}}\ar[rr]\ar@{->>}[dd]
                                    && G\ar[dd]^{\phi} \\
                                    && \\
C_{p^{s_1}}\times\cdots\times C_{p^{s_k}} && \ar[ll]_(.38){\iso} H
}
\]
to which we can apply $\boldsymbol{A}$. In the resulting
diagram
\[
\xymatrix{
\boldsymbol{A}(C_{p^{r_1+s_1}})\otimes\cdots\otimes \boldsymbol{A}(C_{p^{r_k+s_k}})
                                    && \ar[ll]\boldsymbol{A}(G) \\
                                    && \\
\boldsymbol{A}(C_{p^{s_1}})\otimes\cdots\otimes\boldsymbol{A}(C_{p^{s_k}})
    \ar@{ >->}[uu] && \ar[ll]_(.34){\iso}\boldsymbol{A}(H)\ar[uu]_{\phi^*}
}
\]
we see that $\phi^*$ must be monic.

\noindent
(b) Taking the Pontrjagin dual of $\theta$ we obtain
an exact sequence
\[
0 \la \Hom(K,C_{p^\infty}) \xleftarrow{\;\theta^*\;} \Hom(G,C_{p^\infty}),
\]
where $\ds C_{p^\infty} = \colim_r C_{p^r}\subseteq S^1$,
which is an injective $\Z$-module. Since $K$ is a product
of cyclic groups,
\[
K \iso C_{p^{t_1}}\times\cdots\times C_{p^{t_\ell}},
\]
and each projection to a factor $C_{p^{t_i}}$ gives
a homomorphism
\[
\lambda_i\:K \xrightarrow{\;\iso\;} C_{p^{t_1}}\times\cdots\times C_{p^{t_\ell}}
\xrightarrow{\;\ph{\iso}\;} C_{p^{t_i}} \xrightarrow{\;\ph{\iso}\;}
    C_{p^\infty}\xrightarrow{\;\ph{\iso}\;} S^1
\]
we obtain algebra generators for
\[
\boldsymbol{A}(C_{p^{t_1}}\times\cdots\times C_{p^{t_\ell}}) \iso
\boldsymbol{A}(C_{p^{t_1}})\otimes\cdots\otimes\boldsymbol{A}(C_{p^{t_\ell}})
\]
by applying $\lambda_i^*$ to the generators of~$H$ given
in~\eqref{eq:H-lim}. Since each $\lambda_i$ factors through~$G$,
this shows that the algebra generators of $\boldsymbol{A}(K)$
are all in the image of $\theta^*$, therefore $\theta^*$ is
epic.
%
%

The converse statements are easily verified.
\end{proof}

The next result follows easily using standard facts about
commutative groups schemes.
\begin{cor}\label{cor:abelgp-mono-epi}
Suppose that
\[
1\ra G' \xrightarrow{\;f\;} G \xrightarrow{\;g\;} G'' \ra 1
\]
is a short exact sequence of finite abelian $p$-groups.
Then the induced homomorphisms of Hopf algebras
\[
\boldsymbol{A}(G'') \xrightarrow{\;g^*\;} \boldsymbol{A}(G)
                 \xrightarrow{\;f^*\;}\boldsymbol{A}(G')
\]
induce  a short exact sequence of commutative groups schemes
\[
1 \ra \Spec(\boldsymbol{A}(G')) \lra \Spec(\boldsymbol{A}(G))
    \lra \Spec(\boldsymbol{A}(G'')) \ra 1.
\]
\end{cor}

\section{Examples: Honda formal groups and $p$-divisible groups}
\label{sec:Honda}

For each $n\geq1$, there is a $p$-divisible group of
height~$n$, where
\[
H_r = \F[x_r]/(x_r^{q^r}),
\]
where $q$ is a power of $p$ or $\F$ is $\F_p$, $\F_q$
or $\F_{p^\infty}=\bar{\F}_p$. Here the coproduct on
$H_1$ has the form
\[
\psi(x_1) = x_1\otimes1 + 1\otimes x_1
  - \sum_{1\leq i \leq p-1}
\frac{1}{p}\binom{p}{i}x_1^{ip^{n-1}}\otimes x_1^{(p-i)p^{n-1}}.
\]
Furthermore, the natural homomorphisms
\[
H_r \lla H_{r+s} \lla H_s
\]
are given by
\[
H_{r+s} \lra H_r; \quad x_{r+s} \mapsto x_r
\]
and
\[
H_s \lra H_{r+s}\; \quad x_s \mapsto x_{r+s}^{q^r}.
\]
We also have
\[
\soc H_r = \F\{x_r^{q^r-1}\}.
\]
Passing to the limit we obtain
\[
H = \lim_r H_r = \F[\![x]\!],
\]
and the formal group is known as the \emph{Honda formal
group}.

This example is closely connected with the $n$-th Morava
$K$-theory $K_n^*(-)$. More precisely, the $2$-periodic
version has
\[
K_n^* = \F_{p^n}[u,u^{-1}]
\]
which is a graded field with $u\in K_n^{-2}$, and
\[
H_r = K_n^0(BC_{p^r}).
\]
Since $C_{p^r}$ is an abelian group, $BC_{p^r}$ is
a commutative $H$-space and so $K_n^0(BC_{p^r})$ is
a cocommutative graded Hopf algebra over $\F_{p^n}$.
For a suitable choice of compatible generators~$x_r$,
this agrees with the above algebraic example. This
example is also special because the dimension of the
$p$-divisible group is~$1$; in general the dimension
satisfies $1\leq\dim\mathcal{G}\leq h$. Because of the
way Morava $K$-theory and other topological examples
arise, they always give $1$-dimensional $p$-divisible
groups.

\appendix
\section{Some recollections on Frobenius algebras}\label{sec:FrobAlg}

We will use the phrase \emph{Frobenius algebra} to indicate
that an algebra~$A$ has at least one Frobenius structure,
$(A,\Phi)$, where $\Phi\:A\xrightarrow{\;\iso\;}A^*$ is a left
$A$-module isomorphism. It might be better to use Nakayama's
terminology \emph{Frobeniusean} of~\cite{Nakayama:FrobAlgI},
but as remarked by Lam~\cite{Lam:Mod&Rings}*{comments on page~453},
this has fallen out of fashion. For our purposes it is useful
to allow flexibility over the choice of Frobenius structure
on such an algebra. Our usage differs from that of
Koch~\cite{Kock:FrobAlg} who requires the Frobenius structure
as well as the underlying algebra.

Throughout we assume that $\k$ is a field and set
\[
\dim = \dim_\k,\quad \otimes=\otimes_\k,\quad \Hom=\Hom_\k.
\]

We assume that $A$ is a finite dimensional local $\k$-algebra;
here local means that $A$ has a unique maximal left, or
equivalently right, ideal which agrees with the Jacobson radical
$\rad A$, and the quotient $A/\rad A$ is the unique simple
$A$-module, and we will also assume that $A$ is augmented over
$\k$, so that $A/\rad A\iso\k$ is the unique simple $A$-module.
We will use some basic facts about the radical, in particular
it is nilpotent, say $(\rad A)^e=0$ and $(\rad A)^{e-1}\neq0$.
The \emph{socle} of $A$ is the right annihilator of $\rad A$,
\[
\soc A = \{z\in A: (\rad A)z=0\} \supseteq (\rad A)^{e-1},
\]
which is known to be the sum of all the simple left $A$-submodules
of $A$. Indeed, on choosing a minimal set of simple submodules
$V_i$ with $\soc A = V_1+\cdots +V_\ell$, we find that
\[
\soc A = V_1\oplus\cdots\oplus V_\ell.
\]
Of course, for each $i$ there is an isomorphism of $A$-modules
$V_i\iso\k$. Furthermore, $\soc A$ is also the left annihilator
of $\rad A$, and it intersects each non-zero left or right ideal
non-trivially.

Let $A^*=\Hom_\k(A,\k)$ be the $\k$-linear dual of $A$. Then
$A^*$ is a left $A$-module with scalar multiplication~$\.$
given by
\[
a\.f(x) = f(xa)\quad (a,x\in A,\;f\in A^*),
\]
and it is also a right $A$-module with scalar multiplication
\[
(f\.a)(x) = f(ax)\quad (a,x\in A,\;f\in A^*).
\]
In either case $A^*$ is an injective $A$-module.

We recall various aspects of a Frobenius algebra structure
on $A$ and their properties; details can be found
in~\cites{Kock:FrobAlg,Lam:Mod&Rings}. In particular, we
will follow recent tradition in using cobordism diagrams
to express relationships between the various structure
morphisms associated with a Frobenius algebra structure.

For $\Phi\in\Hom_A(A,A^*)\iso\Hom_\k(\k,A^*)\iso A^*$,
the pair $(A,\Phi)$ is a \emph{Frobenius algebra} over
$\k$ if $\Phi$ is an isomorphism. Since $A\iso A^*$,
$A$ is self-injective. It is immediate that if
$(A,\Theta)$ is also a Frobenius algebra then there
is a unit $u\in A^\times$ such that
$\Theta(-) = \Phi(-u^{-1})$, \ie, for all $a\in A$,
\[
\Theta(a) = \Phi(au^{-1}).
\]
This shows that
\begin{prop}\label{prop:FrobAlg-param}
Let $(A,\Phi)$ be a Frobenius algebra. Then the set
of all Frobenius algebras $(A,\Theta)$ is $1$-$1$
correspondence with the set of units $A^\times$.
\end{prop}

Given a Frobenius structure $\Phi$ on $A$, the linear
form $\epsilon=\Phi(1)\in A^*$ has the following
property:
\begin{itemize}
\item
$\ker\epsilon$ contains no non-trivial left ideals.
\end{itemize}
This condition is equivalent to
\begin{itemize}
\item
$\epsilon$ is non-trivial on each simple left $A$-submodule
of~$A$.
\end{itemize}
Notice that if $\dim A>1$, then $\epsilon$ cannot be
a $\k$-algebra homomorphism.

We call $\epsilon$ the \emph{counit} of the Frobenius
algebra $(A,\Phi)$ and sometimes indicate it diagrammatically
by \intextcounitFA.

There is an associated $\k$-bilinear form
$\braket{-|-}\:A\otimes A\lra\k$ given by
\[
\braket{x|y} = \epsilon(xy)
\]
for $x,y\in A$. This is called the \emph{Frobenius
pairing} and is often denoted by the cobordism 
diagram \intextbilinFA. It satisfies the Frobenius 
associativity condition:
\begin{itemize}
\item
for $x,y,z\in A$, $\braket{x|yz} = \braket{xy|z}$.
\end{itemize}
The linear mapping
\[
\lambda\:A\lra A^*;\quad a\lra \braket{-|a}
\]
is non-degenerate since $Aa\subseteq\ker\epsilon$.
By finite dimensionality, the linear mapping
\[
\rho\:A\lra A^*;\quad a\lra \braket{a|-}
\]
is also non-degenerate, and these two linear mappings
give $\k$-linear isomorphisms $A\xrightarrow{\iso}A^*$.
The first of these is actually a left $A$-module isomorphism
$\lambda\:A\xrightarrow{\iso}A^*$, while the second is
a right $A$-module isomorphism. Of course we can recover
$\epsilon$ from $\braket{-|-}$ by using the functional
identities
\[
\epsilon(-) = \braket{-|1} = \braket{1|-}.
\]
Denoting the unit $\k\lra A$ by \intextunitFA, these
amount to the identities of the following cobordism
diagram.
\begin{center}
\begin{texdraw}\setunitscale 0.5
\move(0 0)\counitFA
\htext(40 0){$=$}
\move(80 -15)\unitFA
\move(80 -15)\bilinFA
\htext(130 0){$=$}
\move(170 15)\unitFA
\move(170 -15)\bilinFA
\end{texdraw}
\end{center}

The three structures $\Phi,\phi,\braket{-|-}$ with the above
properties give equivalent information and any one determines
the others, see~\cite{Kock:FrobAlg}*{section~2.2}.

\begin{prop}\label{prop:soc}
If $(A,\Phi)$ is a local Frobenius algebra, then $\dim\soc A=1$
and so there is an isomorphism of $A$-modules $\soc A\iso\k$.
In particular, $\soc A$ is a simple $A$-module.
\end{prop}
\begin{proof}
This amounts to verifying the degree zero Gorenstein condition.
Here we have a unique maximal ideal $\mathfrak{m}\lhd A$ with
$A/\mathfrak{m}\iso\k$ which gives $\k$ a unique $A$-module
structure. Since $A\iso A^*$ is self-injective,
\begin{align*}
\Hom_A(\k,\soc A) = \Hom_A(\k,A)
             &\iso \Hom_A(\k,A^*) = \Hom_A(\k,\Hom_\k(A,\k)) \\
             &\iso \Hom_\k(A\otimes_A\k,\k) \\
             &\iso \Hom_\k(\k,\k) \iso \k.
\end{align*}
Therefore $\dim\soc A=1$.
\end{proof}


Let $B$ be a $\k$-algebra. Recall that there is a sequence
of isomorphisms
\begin{equation}\label{eq:k-A}
B^* = \Hom_\k(B,\k) \xrightarrow{\iso} \Hom_\k(B\otimes_B B,\k)
\xrightarrow{\iso} \Hom_B(B,B^*),
\end{equation}
where we use left $B$-module structures except for in the
tensor product $B\otimes_B B$ where we use the right module
structure on the first factor and the left module structure
on the second factor. Under this composition, $\theta\in A^*$
corresponds to $\Theta\in\Hom_A(A,A^*)$ characterized by
$\Theta(1)=\theta$.
\begin{lem}\label{lem:Frobform-soc}
Suppose that $B$ is a finite dimensional $\k$-algebra for which
$\dim\soc B=1$. Suppose that $\theta\in B^*$ corresponds to
$\Theta\in\Hom_B(B,B^*)$ under the composition of the isomorphisms
of\/~\eqref{eq:k-A}. Then $(B,\Theta)$ is a Frobenius algebra
if and only if\/ $\theta$ is non-trivial on $\soc B$.
\end{lem}
\begin{proof}
Suppose that $(B,\Theta)$ is a Frobenius algebra and note
that $\soc B\iso\k$. Then $\theta\in B^*$ is the corresponding
Frobenius form. Since $\soc B$ is a non-trivial left ideal
in~$B$, $\soc B\nsubseteq\ker\theta$, therefore $\theta$ must
be non-trivial on $\soc B$.

Now suppose that $\theta\in B^*$ is non-trivial on $\soc B$.
If $I$ is a non-trivial left ideal in $B$, then for some
$r\geq1$, $(\rad B)^rI=0$ and  $(\rad B)^{r-1}I\neq 0$. But
then $0\neq(\rad B)^{r-1}I\subseteq\soc B$, and so
$\soc B=(\rad B)^{r-1}I\subseteq I$ since $\dim\soc B=1$.
Hence $\theta$ is not zero on $I$. This shows that $\theta$
must be non-trivial on every non-trivial left ideal and
it follows that $(B,\Theta)$ is a Frobenius algebra.
\end{proof}

\begin{lem}\label{lem:Frobform-unit}
Suppose that $(A,\Theta)$ is a Frobenius algebra with Frobenius
form $\theta\in A^*$. Let $u_0\in\soc A$ and $\theta(u_0)=1$,
and let $u_0,u_1,\ldots, u_{d-1}$ be a $\k$-basis for $A$.
Then given any sequence $t_0,t_1,\ldots, t_{d-1}\in\k$ with
$t_0\neq0$, there is a Frobenius algebra $(A,\Theta')$ whose
Frobenius form $\theta'$ satisfies $\theta'(u_i)=t_i$ for
every~$i$.
\end{lem}
\begin{proof}
By Lemma~\ref{lem:Frobform-soc}, we have $\theta(u_0)\neq0$,
so for simplicity we will assume that $\theta(u_0)=1$. Consider
the dual basis with respect to the associated bilinear form
$\braket{-|-}$, say $v_0,v_1,\ldots,v_{d-1}$ where
\[
\braket{u_i|v_j} = \delta_{i,j}.
\]
In fact, if $z\in \rad A$ then $\braket{z|u_0}=\theta(zu_0)=0$;
hence we can assume that $v_0=1$. Furthermore we then have
$v_i\in\rad A$ for $i\geq 1$. Now define
\[
\theta' = (t_0 + \sum_{1\leq i\leq d-1} t_i v_i)\.\theta \in A^*.
\]
Then for each $r$,
\begin{align*}
\theta'(u_t) &= \theta(t_0u_r + \sum_{1\leq i\leq d-1} t_i u_rv_i) \\
             &= \braket{u_r|t_01 + \sum_{1\leq i\leq d-1} t_i v_i} \\
             &= t_r.
\end{align*}
Then $\theta'\in A^*$ corresponds to $\Theta'\in \Hom_A(A,A^*)$
where $(A,\Theta')$ is a Frobenius algebra.
\end{proof}

We mention a general algebraic result which
appears in a special case in the topological
context of Morava $K$-theory. Let~$A$ and~$B$
be two local Frobenius algebras over the field~$\k$.
Then as $\k$-vector spaces,
\[
\soc A\iso\k\iso\soc B.
\]
When $\alpha\:B\lra A$ is a local algebra
homomorphism, $A$ becomes a left $B$-module
and the above isomorphisms are isomorphisms
of $B$-modules.
\begin{lem}\label{lem:soc-extension}
Given any isomorphism of $A$-modules
$\alpha\:\soc B\xrightarrow{\;\iso\;}\soc A$, there
are extensions to $A$-module homomorphisms $\alpha'\:B\lra A$.
\[
\xymatrix{
\soc B\ar[rr]^{\alpha}\ar[dd]_{\inc} && \soc A\ar[dd]^{\inc} \\
 && \\
B\ar[rr]^{\alpha'} && A
}
\]
\end{lem}
\begin{proof}
Since $A$ is self-injective, the composition
$\inc\circ\alpha\:\soc B\lra A$ has such
extensions to $A$-module homomorphisms $B\lra A$.
\end{proof}

We also note the following.
\begin{lem}\label{lem:A->Bsoc}
Suppose that the algebra homomorphism $\alpha\:B\lra A$
is non-trivial on $\soc B$. Then there is an element
$a\in A$ for which
\[
a\alpha\soc B = \soc A.
\]
\end{lem}
\begin{proof}
By assumption $\dim\alpha\soc B=1$, so
$\{0\}\neq A\alpha\soc B\ideal A$ and
$(A\alpha\soc B)\cap\soc A\neq\{0\}$. Hence
we must have $\soc A\subseteq A\alpha\soc B$.
For any $a\in A$ satisfying
$\soc A\cap a\alpha\soc B\neq\{0\}$, as
$\dim\soc A=1$ we have $a\alpha\soc B=\soc A$.
\end{proof}

\begin{prop}\label{prop:Gysin}
Let $\alpha\:B\lra A$ be an $A$-module homomorphism
for which $\alpha\soc B = \soc A$, and let $\lambda$
be a Frobenius form on $A$. Then $\lambda\circ\alpha$
is a Frobenius form on~$B$.
\end{prop}
\begin{proof}
Since
\[
\lambda\circ\alpha\soc B = \lambda\soc A \neq0,
\]
this is immediate.
\end{proof}
\begin{prop}\label{prop:Gysin-construction}
Suppose that $\lambda_A$ and $\lambda_B$ are Frobenius
forms for $A$ and $B$ respectively. Then there is an
$A$-module homomorphism $\alpha\:B\lra A$ for which
$\alpha\soc B = \soc A$.
\end{prop}
\begin{proof}
The Frobenius forms give rise to an isomorphism of
$A$-modules
\[
A^* \xrightarrow{\;\iso\;} A;\quad \lambda_A \mapsto 1,
\]
and an isomorphism of $B$-modules (and hence of $A$-modules)
\[
B \xrightarrow{\;\iso\;} B^*;\quad 1\mapsto \lambda_B.
\]
Composing with $\theta$ we obtain a homomorphism of $A$-modules
$\alpha$,
\[
\xymatrix{
B \ar[r]_{\iso}\ar@/^19pt/[rrrr]^{\alpha}\ar@/_19pt/[rrr]_{\alpha'}
  & B^*\ar[rr]_{\theta^*} & & A^*\ar[r]_{\iso} & A
}
\]
where $\theta^*(f) = f\circ\theta$. Let $\alpha'\:B\lra A^*$
be the intermediate composition, given by
\[
\alpha'(b)(a) = \lambda_B(\theta(a)b)
\]
for $b\in B$ and $a\in A$. In particular, if $z\in\soc B$
and $w\in\rad A$, then by definition of $w\alpha'(z)$ we
have
\[
(w\alpha'(z))(a) = \lambda_B(\theta(aw)z)
                 = \lambda_B(\theta(a)\theta(w)z)
                 = 0,
\]
since $\theta$ is local and therefore $\theta(w)z=0$. It
follows that $w\alpha(z)=0$ for all $w\in\rad A$ and so
$\alpha\soc B\subseteq\soc A$. If $z\neq0$,
\[
\alpha'(z)(1) = \lambda_B(\theta(1)z) = \lambda_B(z) \neq 0
\]
since $\dim\soc B=1$ and $\lambda_B$ is non-trivial on
$\soc B$. Hence we have $\alpha\soc B = \soc A$.
\end{proof}

\end{document}

%% file: bukser.tex
\def\fatness{1.2}
\everytexdraw{%
    \drawdim pt \linewd 0.5 \textref h:C v:C
  \arrowheadsize l:4 w:3
  \arrowheadtype t:V
}

\newcommand{\counitFA}{%
    \bsegment
        \realmult {\fatness} {2.5} \xrad
        \realmult {\fatness} {6.25} \yrad
        \move (0 0)
        \move (20 0)
        \move (0 0) \lellip rx:{\xrad} ry:{\yrad}
        \realadd { 0} { \yrad} \inB
        \realadd { 0} {-\yrad} \inA
        \move (0 {\inB}) \rlvec (8 0) 
        \move (0 {\inA}) \rlvec (8 0) 
      \move (8 {\inA}) \clvec (20 {\inA})(20 {\inB})(8 {\inB}) 

        \savepos (0 0)(*ex *ey)
    \esegment
    \move (*ex *ey)
}
\newcommand{\unitFA}{%
    \bsegment
        \realmult {\fatness} {2.5} \xrad
        \realmult {\fatness} {6.25} \yrad
        \move (0 0)
        \move (-20 0)
        \move (0 0) \lellip rx:{\xrad} ry:{\yrad}
        \realadd { 0} { \yrad} \inB
        \realadd { 0} {-\yrad} \inA
        \move (0 {\inB}) \rlvec (-8 0) 
        \move (0 {\inA}) \rlvec (-8 0) 
      \move (-8 {\inA}) \clvec (-20 {\inA})(-20 {\inB})(-8 {\inB}) 

        \savepos (0 0)(*ex *ey)
    \esegment
    \move (*ex *ey)
}

\newcommand{\bilinFA}{%
  \bsegment
        \realmult {\fatness} {2.5} \xrad
        \realmult {\fatness} {6.25} \yrad
        \move (0 0) \lellip rx:{\xrad} ry:{\yrad}
        \move (0 30) \lellip rx:{\xrad} ry:{\yrad}
        \realadd {30} { \yrad} \inD
        \realadd {30} {-\yrad} \inC
        \realadd { 0} { \yrad} \inB
        \realadd { 0} {-\yrad} \inA
                \realmult {\yrad} {2} \diam
        \realadd {24} {\diam} \storsving
        \move ({\storsving} 0) 

        \move (0 {\inB}) \clvec (18 {\inB})(18 {\inC})(0 {\inC}) 
        \move (0 {\inA}) \clvec ({\storsving} {\inA})({\storsving} {\inD})(0 {\inD}) 

        \savepos (0 30)(*ex *ey)
\esegment
\move (*ex *ey)
}

\newcommand{\intextbilinFA}{%
  \raisebox{-2.2pt}{\begin{texdraw}
  \drawdim pt \linewd 0.3 \setunitscale 0.3
  \bilinFA
  \move (-6 0)
  \move (30 0)
  \end{texdraw}}
}

\newcommand{\intextunitFA}{%
  \raisebox{0.9pt}
  {\begin{texdraw}
  \drawdim pt \linewd 0.3 \setunitscale 0.45
  \unitFA
  \move (6 0)
  \end{texdraw}}
}
\newcommand{\intextcounitFA}{%
  \raisebox{0.9pt}
  {\begin{texdraw}
  \drawdim pt \linewd 0.3 \setunitscale 0.45
  \counitFA
  \move (-6 0)
  \end{texdraw}}
}

\newcommand{\onedot}[1]{
  \bsegment
    \move (0 0) \fcir f:0 r:2
    \esegment
}
